\documentclass
{article}

\usepackage{latexsym}
\usepackage{amsfonts, amsmath, amscd}
\usepackage[psamsfonts]{amssymb}
\usepackage[usenames]{color}
\usepackage[latin1]{inputenc}
\usepackage{amsmath,amsfonts,amssymb}

\newtheorem{lemma}{Lemma}[section]
\newtheorem{theorem}[lemma]{Theorem}
\newtheorem{cor}[lemma]{Corollary}
\newtheorem{fact}[lemma]{Fact}

\newtheorem{definition}[lemma]{Definition}
\newtheorem{question}[lemma]{Question}
\newtheorem{pro}[lemma]{Proposition}
\newtheorem{problem}[lemma]{Problem}

\newtheorem{exa}[lemma]{Example}
\newtheorem{rmk}[lemma]{Remark}

\newenvironment{proof}{\noindent {\sl Proof.}}{$\Box$ \bigskip}

\newcommand{\T}{{\mathbb T}}
\newcommand{\cont}{{\mathfrak c}}
\newcommand{\Z}{{\mathbb Z}}

\newcommand{\N}{{\mathbb N}}

\newcommand{\aaa}{{\mathbf a}}
\newcommand{\uuu}{{\mathbf u}}
\newcommand{\vvv}{{\mathbf v}}

\newcommand{\Tv}{\mathcal{T}_\mathbf{v}}
\newcommand{\Tu}{\mathcal{T}_\mathbf{u}}

\newcommand{\Tub}{{\tau}_\mathbf{u}}
\newcommand{\Tvb}{{\tau}_\mathbf{v}}

\def\T{{\mathbb T}}

\def\Z{{\mathbb Z}}
\def\N{{\mathbb N}}
\def\R{{\mathbb R}}

\def\(SB){{\mathfrak B}}

\makeindex

\begin{document}


\title{Characterizing sequences for precompact group topologies}

\author{D. Dikranjan,
\\{\footnotesize Dipartimento di Matematica e Informatica,}
\\{\footnotesize Universit\`{a} di Udine, Via delle Scienze  206,}
\\{\footnotesize  Localit\` a Rizzi}
\\{\footnotesize  33100 Udine, Italy}
 \\ {\footnotesize {\tt dikran.dikranjan@uniud.it}}
\and S. S. Gabriyelyan
\\{\footnotesize Department of Mathematics,}
\\{\footnotesize Ben-Gurion University of the Negev,}
\\{\footnotesize  Beer-Sheva, P.O. 653,}
\\{\footnotesize  Israel}
\\{\footnotesize  e-mail: saak@math.bgu.ac.il}
\and V. Tarieladze\thanks{The third  named author was partially supported  by grant GNSF/ST$09_{-}99_{-}3-104$}
\\{\footnotesize  Niko Muskhelishvili Institute of Computational Mathematics}
\\{\footnotesize  of the Georgian Technical University}
\\{\footnotesize  8, Akuri str. 0160 Tbilisi}
\\{\footnotesize   Georgia}
\\{\footnotesize  e-mail: vajatarieladze@yahoo.com}
}
\date{}

\maketitle

\begin{abstract}

A precompact group topology $\tau$ on an abelian group $G$ is called {\em single sequence characterized} (for short,
{\em ss-characterized}) if there is a sequence $\uuu= (u_n)$ in $G$
such that  $\tau$ is  the finest  precompact group topology on $G$ making  $\uuu=(u_n)$ converge to zero.
It is proved that a metrizable precompact abelian group  $(G,\tau)$ is $ss$-characterized iff it is countable. For every
metrizable precompact group topology $\tau$ on a  countably  infinite  abelian group $G$  there exists a group topology $\eta$ such
 that $\eta$ is strictly finer than $\tau$ and the groups $(G,\tau)$ and $(G,\eta)$ have the equal Pontryagin dual groups.
We give a complete description of all $ss$-characterized precompact abelian groups modulo countable $ss$-characterized groups from which we derive:

(1) No infinite pseudocompact abelian group  is $ss$-characterized.

(2)  An $ss$-characterized  precompact abelian group is hereditarily disconnected.
\end{abstract}

\medskip



\section{Introduction} \label{sec1}

\subsection{Notations and terminology}\label{NoTe}

Usually, all considered groups will be abelian and additive notation will be used. All considered topological groups are Hausdorff.
For a subset $A$ of a group $G$ denote by  $\langle A\rangle$ the subgroup of $G$ generated by  $A$.
An abelian group $G$ endowed with the  discrete topology is denoted by $G_d$.

 For a topological group $(G,\tau)$ we denote by   $\mathcal{N}(G,\tau)$  the filter of all  neighborhoods of  the neutral element of  $G$.
The group of all continuous characters of $(G,\tau)$ we denote by $G^\ast =(G,\tau)^\ast$.
The group $G^\ast $ endowed with the compact open topology is called {\it the Pontryagin dual} of $(G,\tau)$ and it is denoted by $G^\wedge$.
For $x\in G$ we define a mapping $\hat x:G^\wedge \to \T$ by $\hat
x(\chi)=\chi(x)$ for $\chi\in G^\wedge$.  Then $\hat x\in G^{\wedge\wedge}$ for each $x\in G$ and the mapping $G\buildrel{\alpha_G} \over\longrightarrow
G^{\wedge\wedge}$, $\alpha_G(x)=\hat x$, is a group homomorphism, called  {\it the canonical homomorphism}.
The group $G$ is called {\it maximally almost periodic}, for short a $MAP$-group, if $G^\ast $ separates points of $G$, i.e., $\alpha_G$ is injective.
We denote by MAP the class of all $MAP$-groups.
A subgroup $H$ of $G$
is called {\it dually closed} if for every $x\in G\setminus H$ there exists $\chi\in G^\ast$ such that
$\chi(H)=\{0\}$ and $\chi(x)\ne 0$, and it is called {\it dually embedded} if for every $\varphi \in H^\ast$ there exists $\chi\in G^\ast$ such
that $\chi|_{H}=\varphi$.

Let us recall \cite{CR} that every precompact  group topology $\sigma$ on an abelian group $G$ is uniquely determined by the set $G^\ast$ of all continuous characters of $(G,\sigma)$.
Namely, $\sigma$ has the form $T_{G^\ast}$, where $T_{G^\ast} := \sigma(G,G^\ast)$ is the coarsest group topology on $G$ with respect to which all members of $G^\ast$
are continuous. Conversely, if $H$ is an arbitrary dense subgroup of the Pontryagin dual compact group $(G_d)^\wedge$,
then $T_H :=\sigma(G,H)$ is a precompact group topology on $G$ such that $(G, T_H)^\ast =H$ (see Fact \ref{f31}).

 For a topological group $(G,\tau)$  we write  $\tau^+:=\sigma(G, \Gamma)$, where  $\Gamma=(G,\tau)^{\ast}$.
Clearly $\tau^+\le \tau$. The topology $\tau^+$ is often called  {\it the Bohr
modification} of $\tau$.

 A subset $A$ of a topological abelian group $G$ is called {\it quasi-convex} if for every $g\in G\setminus A$ there exists $\chi\in G^\ast$ such that
\[
\chi(A)\subset \T_+,\,\,\,\text{but}\,\,\chi(x)\not\in \T_+ \, ,
\]
where  $\T_+ $ is  the image of the segment $[-\frac{1}{4},\frac{1}{4} ]$ with respect to the natural quotient map $\R\rightarrow \T$. If $A$ is a subset of $G$, then the set
\[
\mathrm{qc}(A) := \bigcap \{ Q : A\subseteq Q \mbox{ and } Q \mbox{ is quasi-convex}\}
\]
is quasi-convex in $G$ and it is called {\it the quasi-convex hull of} $A$.
A topological group $(G,\tau)$ (as well as its topology $\tau$) is called {\it locally quasi-convex} if it has a basis
consisting quasi-convex subsets of $(G,\tau)$. We denote by $\rm{LQC}$ the class of all   locally quasi-convex groups. Clearly, every $LQC$-group is also $MAP$.

For a  MAP abelian topological group $(G,\tau)$ (cf. \cite[6.18]{Aus}) the collection $\{{\rm{qc}}(V):V\in \mathcal{N}(G,\tau)\}$ is a base of a locally quasi-convex group topology $(\tau)_{\rm{lqc}}$ called {\it the locally quasi-convex  modification} of $\tau$.
The topology $(\tau)_{\rm{lqc}}$ is the finest one among the locally
quasi-convex group topologies on $G$ which are coarser than $\tau$.
Note  that every precompact group topology is  locally quasi-convex.

If $\tau$ and $\eta$ are group topologies on a group $G$, $\eta$ is said to be ({\em strongly}) {\em compatible}  with $\tau$ if $(G,\eta)^{\ast}=(G,\tau)^{\ast}$ ($(G,\tau)^{\wedge}=(G,\eta)^{\wedge}$, resp.). If $(G,\tau)$ is $MAP$, then $\tau^+$ and $(\tau)_{\rm{lqc}}$ are compatible with $\tau$.
Moreover, if $\tau^+$-compact sets are $\tau$-compact, then $\tau^+$ is strongly compatible with  $\tau$.

Let $\mathcal G$ be a class of  topological abelian groups (for example MAP or $\rm{LQC}$). Following \cite{DEV}, a topological group $(G,\mu)$  is called a  {\em Mackey group}  in $\mathcal G$  (or a $\mathcal G$-{\it Mackey group}) if $(G,\mu)\in \mathcal G$ and if $\nu$ is a compatible group topology with  $\mu$ such that  $(G,\nu)\in \mathcal G$, then $\nu\le \mu$.

For an abelian topological  group $X$ and a sequence $\uuu = ( u_n)$ in its dual group one puts:
\begin{equation} \label{11}
s_\uuu (X):= \{x\in X: u_n(x) \to 0 \mbox{ in }\T \}.
\end{equation}
The subgroups of the form (\ref{11}) will play a crucial role in our considerations.

\subsection{Main results}

The article is devoted to investigation of the following notions  introduced in \cite{BDMW1, DMT}:
\begin{definition} \label{d01}
Let  $\mathbf{u}=(u_n)$ be a sequence in an abelian group $G$.
\begin{itemize}
  \item[{\rm (a)}] {\rm \cite{BDMW1}} If there exists a precompact group topology making  $\uuu=(u_n)$ converge to $0$, we call $\uuu$ a {\em $TB$-sequence}.
For a $TB$-sequence $\uuu$ there exists the {\em finest}  precompact group topology $\Tub$ making  $\uuu=(u_n)$ converge to $0$.
\item[{\rm (b)}] {\rm \cite{DMT}} For a precompact group topology $\tau$ on $G$ we say that the sequence $\uuu= (u_n)$ in $G$ {\em characterizes} $\tau$ (and  $\uuu$ is {\em a characterizing sequence} for $(G,\tau))$, if $\uuu$ is a $TB$-sequence and $\tau = \Tub$.
\item[{\rm (c)}]  A precompact group $(G,\tau)$, as well as its topology $\tau$, is called {\em single sequence characterized} (for short,
{\em $ss$-characterized}) if there exists a sequence of elements of $G$ which characterizes $\tau$.
\end{itemize}
\end{definition}

Hence, Definition \ref{d01} defines  a correspondence
\begin{equation} \label{01}
\uuu \mapsto \Tub
\end{equation}
between $TB$-sequences and precompact topologies. The necessity to understand better the nature of the
correspondence (\ref{01}) motivates   the following general question:

\begin{problem} \label{p00}
Describe all $ss$-characterized precompact abelian groups.
\end{problem}

Let us note that a {\it set-theoretic} description of the topology of an $ss$-characterized group was obtained in \cite{DMT} (see the more general Fact \ref{f1}):

\begin{fact} {\rm \cite{DMT}}\label{f0}
For every $TB$-sequence $\uuu$ in an abelian group $G$ one has  $(G, \Tub)^\ast = s_\uuu \left((G_d)^\wedge\right)$.
\end{fact}
Nevertheless, till now nothing was known about the {\it topological} properties of the groups in the codomain of (\ref{01}). That is, we did not know even answers to the following questions. Which metrizable precompact groups are $ss$-characterized?  Which other natural classes of precompact group topologies are $ss$-characterized? Which totally disconnected precompact or pseudocompact group topologies are in the codomain  in (\ref{01})? What we can say about topological properties of $ss$-characterized precompact groups as being a $k$-space? The main goal of  the article is to give, among others, complete answers to these questions (see Theorem A and Corollaries C1-C3). Also we give some applications of the obtained results (see Corollary A and Theorem \ref{t42}). Our solution is essentially supported by an appropriate reduction of Problem \ref{p00} to the case of countable groups (see Theorem C).

The next theorem shows that the codomain in (\ref{01}) covers all metrizable precompact topologies on  countably infinite abelian groups:

\medskip

\noindent {\bf Theorem A.}
{\em A metrizable precompact group $G$ is $ss$-characterized if and only if $G$ is countable. Moreover, every characterizing sequence of $G$ generates a finite index subgroup of $G$.}
\medskip

For instance,  the sequence $\mathbf{p} = (p^n)$, where $p$ is prime, is a $TB$-sequence generating  the $p$-adic topology of $\Z$ (see Example \ref{Eggl}(b) below).

As a corollary of Theorem A we show that precompact  metrizable countable groups can not be MAP-Mackey.

\medskip
\noindent {\bf Corollary A.}
{\em Let $(G,\tau)$ be a countably infinite precompact metrizable group. Then on  $G$ there exists a group topology $\eta$ which is strictly finer than $\tau$ and it is strongly compatible with $\tau$. In particular, $(G,\tau)$ is not a $MAP$-Mackey group.}
\medskip

Let us note that Corollary A is the best possible in the following sense: there are countably infinite precompact
metrizable groups {\it which are} LQC-Mackey \cite{BTV} (see also \cite{DDMT}). However, there are
also countably infinite precompact metrizable groups {\it which are not } LQC-Mackey \cite{LD}; see also \cite{DEV}, where a wide class
of (uncountable) precompact metrizable non-LQC-Mackey groups is described. Note also that $\T$ is a LQC-Mackey group, however it
remains unknown whether $\T$ is a MAP-Mackey group as well.

Let $\uuu$ be a $TB$-sequence in an abelian group $G$. Clearly, the {\it countable} subgroup $\langle\uuu\rangle$ of $G$ generated by $\uuu$ must play a crucial role to a solution of Problem \ref{p00}. The first natural step in the understanding of $ss$-characterized topologies is to reduce the general problem to the case of countable groups:
\begin{problem} \label{p01}
Describe all countable $ss$-characterized precompact abelian groups.
\end{problem}
This problem is especially  important in the light of Theorem A. Moreover, a study of the dual group of a countable $ss$-characterized group is interesting from the duality theory point of view
(see Problem section in \cite{Ga3}).

It turns out that the reduction of Problem \ref{p00} to Problem \ref{p01}, that is one of the main goals of the article (see Theorem C), is non-trivial and needs two new notions which are of independent interest.

\begin{definition}
A subgroup $H$ of a topological abelian  group $G$ is said to be
{\em $B$-embedded} if every (algebraic) character $\chi:G/H \to \T$ is continuous in the quotient topology of $G/H$.
\end{definition}

The choice of the term $B$-embedded is explained by the fact that the Bohr topology of  $G/H$  coincides with the Bohr topology of
$(G/H)_d $. Obviously, $B$-embedded subgroups are closed (see below Proposition 3.5).
It should be noted that this notion is strongly related to the existing notion of an $h$-embedded subgroup ($H$ is an $h$-embedded subgroup of a topological abelian  group $G$ if every homomorphism $\chi: H \to \T$ extends to a continuous character of $G$) \cite{T}.

Let $H$ be a subgroup of an abelian group $G$ and $\tau$ be a Hausdorff group topology on $H$. Then the set ${\mathcal N}(H,\tau)$  is a local base of a Hausdorff group  topology $\bar \tau$ on $G$.  Clearly,  $\bar \tau$ is the {\it finest} group topology on $G$ in the class of all Hausdorff group topologies on $G$ extending $\tau$. (The fact that $G$ is abelian is important, in the non-abelian case this construction does not work even if $H$ is a subgroup of index 2 of $G$ \cite{DS0}.) In particular, every subgroup $N$ of $G$ containing $H$ is open (so clopen) in $\bar \tau$.

In the above situation, if the subgroup $H$ of $G$ carries a precompact  topology $\tau$, it is natural to ask whether
there is any precompact topology on $G$ that extends $\tau$. Moreover, motivated by the case of Hausdorff group topologies, it is
natural to ask whether a finest precompact  topology extending $\tau$ exists in this case.

\begin{definition}
Let  $H$  be a subgroup  of an abelian group $G$  and let  $\tau$ be a precompact topology on $H$. A precompact topology  $\tau^\ast$ on $G$ is called a {\rm  finest
precompact extension} of $\tau$  on $G$ if  $\tau^\ast|_H = \tau$ and $\tau' \leq \tau^\ast$ for every precompact topology $\tau'$ on $G$ such that $\tau'|_H =\tau$.
\end{definition}

Obviously, the finest precompact extension is  unique, whenever it exists. The existence of {\em some} precompact extension was established in \cite[Theorem 10.1]{DS}. Moreover, a careful analysis of the proof of \cite[Theorem 10.1]{DS} shows that the specific extension constructed in \cite[Theorem 10.1]{DS} is actually  the {\it finest} precompact extension of $\tau$.
The next theorem provides an explicit description as well as other descriptions of this finest precompact extension $\tau^\ast$.
Its proof, given in \S3, contains also the crucial steps from the proof of  \cite[Theorem 10.1]{DS}, for reader's convenience.
\\

\noindent {\bf Theorem B.} {\em Let $G$ be an  abelian group and $H$
be an arbitrary  subgroup of $G$. Then
\begin{itemize}
\item[(a)] If $\zeta$ is a precompact  topology on $H$, then $(\bar{\zeta})^+$ is the finest precompact extension of $\zeta$.
\par
\item[(b)] For a precompact  topology $\tau$ on $G$ and $\zeta= \tau|_H$, the following assertions are equivalent:
 \begin{itemize}
   \item[{\rm (b$_1$)}] $\tau$ is the finest precompact extension of  $\zeta$.
   \item[{\rm (b$_2$)}] $H$ is $B$-embedded in $(G,\tau)$.
   \item[{\rm (b$_3$)}] Let $j: H_d \to G_d$ be the identity map. Then $ (G,\tau)^\wedge= (j^\wedge)^{-1}\left( (H,\zeta)^\wedge\right)$.
   \item[{\rm (b$_4$)}] $\tau  = \left(\overline{{\zeta}}\right)^+$.
\end{itemize}
\end{itemize}}

Let $\uuu$ be a $TB$-sequence in an abelian group $G$ and $H$ be an arbitrary subgroup containing $\langle\uuu\rangle$. Then $\uuu$ is a $TB$-sequence in $H$. We denote by $\Tub (H)$ the finest precompact group topology on $H$ in which $\uuu$ converges to zero. The next proposition is used in the proof of Theorem C.
\begin{pro} \label{p44}
Let $\uuu$ be a $TB$-sequence in an abelian group $G$ and $H$ be an arbitrary subgroup of $G$ containing $\langle\uuu\rangle$. Then $\Tub$ is the maximal precompact extension of $\tau_\mathbf{u}(H)$. In particular,  $\Tub|_{H} =\Tub (H)$.
\end{pro}

Since the maximal precompact extension is unique,  Proposition \ref{p44} implies:

\begin{cor} \label{c45}
Let $\uuu$ and $\vvv$ be  $TB$-sequences in an abelian group $G$ and $H$ be an arbitrary subgroup of $G$
containing $\uuu$ and $\vvv$. Then $\Tub =\Tvb$ if and only if $\Tub(H) =\Tvb(H)$.
\end{cor}

The following theorem describes all $ss$-characterized precompact abelian groups by modulo {\it countable} $ss$-characterized groups and is the reduction principle to the countable case:\\

\noindent {\bf Theorem C.} {\em For a precompact  abelian group $(G, \tau)$ the following are equivalent:

\begin{enumerate}
  \item[{\rm (i)}] $(G, \tau)$ is $ss$-characterized.
  \item[{\rm (ii)}] $(G, \tau)$ has a countable $B$-embedded $ss$-characterized subgroup $H$.
\end{enumerate}
}

Compact characterized groups are finite, as the following more precise corollary shows.
Let us recall first that a topological group $G$ is said to be {\em pseudocompact} if every real-valued continuous function on $G$ is bounded (so that every compact group is pseudocompact).

\medskip

\noindent {\bf Corollary C1.} {\em A pseudocompact Hausdorff  group $G$ is $ss$-characterized if and only if $G$ is finite.}

\medskip

Other two compact-like properties missed by the uncountable $ss$-characterized precompact groups are provided by the next two corollaries:

\medskip

\noindent {\bf Corollary C2.} {\em
 Let $(G,\tau)$ be $ss$-characterized precompact  group  which  is  a $k$-space too. Then $G$ is countable and sequential.
}
\medskip

According to \cite{DS1}, a topological group $G$ is an {\em Arhangel'ski\u \i \ group}, if the weight $w(G)$ of $G$ does not exceed of the cardinality  $|G|$ of $G$, i.e., $w(G)\leq |G|$. Compact groups are Arhangel'ski\u \i \ groups, more generally, locally minimal groups are Arhangel'ski\u \i \ groups \cite{ACDD} (it is known that locally compact groups, as well as minimal groups are locally minimal).

\medskip

\noindent {\bf Corollary C3.} {\em A precompact Arhangel'ski\u \i \ group is $ss$-characterized if and only if $G$ is metrizable (and countable).}

\medskip

The article is organized as follows. In Section 1 we give all necessary definitions and examples explaining  our main notions.
In Section 2 we prove the sufficiency in Theorem A and we deduce from it Corollary A.
In \S 3.1 we prove Theorem B. Theorems A and C, and Corollaries C1, C2 and C3 are proved in \S  3.2.

\subsection{Background on $T$- and $TB$-sequences}
\label{HEN}

The question of when a given sequence in an abelian group may converge to $0$ in some Hausdroff topology is of independent interest.

\begin{definition} {\rm \cite{PZart}} \label{d02}
Let $G$ be an abelian group and let  $\mathbf{u}=(u_n)$ be a sequence in $G$.
If  there exists a Hausdroff group topology making  $\uuu=(u_n)$ converge to $0$, then we call $\uuu$ a {\em $T$-sequence}.
For a $T$-sequence $\uuu$ there exists the {\em finest} Hausdroff group topology $\Tu$ that makes  $\uuu=(u_n)$ converge to $0$.
\end{definition}
This topology was first considered by Graev \cite{G}, and later by Protasov and Zelenyuk \cite{PZart,PZ}. $T$-sequences were thoroughly studied in \cite{PZart}.
Clearly, every $TB$-sequence is also a $T$-sequence, but the converse in general is not true. While it is quite hard to check whether a given sequence is a $T$-sequence,
the case of $TB$-sequences is much easier to deal with due to a very simple criterion (see Fact \ref{Crit}).

\begin{rmk} \label{zprot}
{\em Let us note the following properties of topologies of the form $\Tu$ discovered in \cite{PZart,PZ}. Let $\uuu =(u_n)$ be an arbitrary nontrivial $T$-sequence
in an abelian group $G$ (i.e., $u_n \not= 0$ for infinitely many indices). Then  $(G,\Tu)$ is a complete sequential (and hence a $k$-space) but not a Fr\'{e}chet-Urysohn group. In particular, $(G,\Tu)$ is not metrizable. } 
 \end{rmk}

The next fact will be essentially used in the proof of Corollary A:
\begin{fact}\label{fPre}{\rm \cite[2.3.12]{PZ}}
Let $\uuu$ be a $T$-sequence in an abelian group $G$. Then $(G,\Tu)$ is not precompact. In particular, if $\uuu$ is a $TB$-sequence in $G$, then $\Tub\not= \Tu$.
\end{fact}


The following criterion gives a simple dual condition on a sequence to be a $TB$-sequence:
\begin{fact}\label{Crit}{\rm \cite{BDMW1}}
A sequence $\uuu$ in an abelian group $G$ is a $TB$-sequence if and only if the subgroup  $s_\uuu((G_d)^{\wedge})$ is dense in  the compact group $ (G_d)^{\wedge}$.
\end{fact}

A good supply of examples of $TB$-sequences can be found in the cyclic group $G = \Z$, where the criterion \ref{Crit} works especially well,
since the dense subgroups of $\T = \Z^{\wedge}$ are exactly the infinite subgroups of $\T$.

\begin{exa} \label{Eggl}
{\rm Let $\mathbf{a}=(a_n)$ be a sequence  of positive integers and $q_n= a_{n+1}/a_n$ for every $n$.
\begin{itemize}
\item[(a)] If $q_n \to \infty$, then
\begin{itemize}
\item[(a1)] $|s_\mathbf{a}(\T)|=\cont$. (This  was proved by Egglestone \cite{Egg}; see also \cite{BDMW1}).
\item[(a2)] $\aaa$ is a $TB$-sequence by (a1) and Fact \ref{Crit}.
\item[(a3)] The dual of the (countably infinite) precompact   group $(\Z,\tau_{\mathbf{a}})$ has the cardinality $\cont$; in particular,
$(\Z,\tau_{\mathbf{a}})$ is not metrizable.
\end{itemize}
  \item[(b)] For a prime number $p$  let us consider the sequence $\mathbf{p} = (p^n)$. According to \cite{A},
  the subgroup $s_\mathbf{p} (\T)$ coincides with the $p$-torsion subgroup $\Z(p^\infty)$ of $\T$, so it is infinite. Hence, by Fact \ref{Crit}, the sequence
  $\mathbf{p}$ is a $TB$-sequence (note that the sequence of ratios $(q_n)$ now is bounded, actually, constant) and $\tau_\mathbf{p}$ is the $p$-adic topology of $\Z$.
  Note that $(\Z, \tau_\mathbf{p})$ is metrizable (cf. Remark \ref{zprot}).
\item[(c)] Let  $\mathbf{f}=(f_n)$ be the sequence defined by $f_0= f_1= 1$ and $f_{n+2}=f_{n+1} + f_n$ for all $n\in \N$. This is the  celebrated {\it Fibonacci's sequence}.
 It was proved by Larcher \cite{La}  (see also \cite{BDMW1}) that $s_\mathbf{f}(\T)=\langle \alpha\rangle$, where $\alpha$ is the positive solution of the equation $x^2 -x-1 =0$ (namely, {\it the Golden Ratio}) taken modulo 1 (i.e., as an element of $\T$). In particular, $s_\mathbf{f}(\T)$ is infinite, hence dense (as $\alpha$ is irrational). So, by Fact \ref{Crit}, $\mathbf{f}$ is a $TB$-sequence.
A significant generalization of this fact to recursively defined sequences in $\Z$ of higher degree can be found in \cite{BDMW3}.
\item[(d)] Let us show that for every non-zero polynomial $P(x)\in \mathbb{Z}[x]$, the sequence of integers $\uuu =(P(n))_{n\in\N}$ is not a $T$-sequence in $\Z$.
 The proof goes by induction on the degree $\deg(P)$ of $P(x)$. If $\deg(P)=0$, then $P(x)= A\not=0$ and hence $u_n = A\not= 0$ for every $n\in\N$.
 Clearly, $\uuu$ is not
a $T$-sequence. Assume that  for every polynomial $P(x)\in
\mathbb{Z}[x]$ with $\deg(P)\leq m$ the sequence $\uuu =(P(n))$ is
not a $T$-sequence. Let $P_0(x)\in \mathbb{Z}[x]$ with
$\deg(P_0)\leq m+1$. Assume that  $\tau$ is a Hausdorff group
topology on $\Z$ with $\uuu=(P_0
(n))_{n\in\N}\buildrel{\tau}\over\longrightarrow 0$.
 Then also $v_n := u_{n+1} - u_n = P_0 (n+1) - P_0 (n) \to 0$  in $\tau$. Thus $\vvv$ is a $T$-sequence defined by
 the polynomial $P(x):= P_0 (x+1) - P_0 (x)$ of degree $m$,
 a contradiction. (The case of the sequence $(n^2)_{n\in\N}$ was settled in a different way in \cite{CC}).
\end{itemize} }
\end{exa}

It is worth mentioning that the subgroups of the form (\ref{11}) of the torus $\T$ characterize the so called topologically torsion elements.
Using the sequence $\mathbf{p} = (p^n)$ from Example \ref{Eggl}(b), Braconnier \cite{Bra} and Vilenkin \cite{Vi}
defined {\em topologically $p$-torsion elements} for an arbitrary locally compact abelian groups, namely the elements of the
subgroup $t_\mathbf{p}(G):= \{x\in G: p^nx \to 0\}$ of $G$. Using the sequence $\uuu = (n!)$,
Robertson (see \cite{A}) defined {\em topologically torsion elements}
by $t_\uuu(G):= \{x\in G: n! x \to 0\}.$
As noticed in  \cite{D}, the notion of {\em topologically $\uuu$-torsion element} can be extended
to any sequences $\uuu$ of integers and  an arbitrary topological abelian group $G$ by letting
\begin{equation} \label{ee}
t_\uuu(G):= \{x\in G: u_nx \to 0\}
\end{equation}
(for sequences with $u_n| u_{n+1}$ this can be found already in
\cite[Chapter 4, Notes]{DPS}). The subgroups of $\T$ of this form where studied by Borel \cite{Borel}, who proved that every countable
subgroup of $\T$ has the form  $t_\uuu(\T)$ for an appropriate $\uuu$. Later on,
this theorem was reproved in \cite{BDS} (with a gap in the case of torsion subgroups), where the sequence $\uuu$ ensuring $H =
t_\uuu(G)$ was called a {\em characterizing sequence} for $H$.
\begin{rmk}
{\em For a sequences $\uuu$ of integers and a topological group $G$, a non-torsion element $x\in G$ generates a subgroup $\langle x \rangle$
 algebraically isomorphic
to $\Z$. Since $x\in t_\uuu(G)$ precisely when $x\in t_\uuu( \langle
x \rangle)$, this shows that $\uuu$ {\em is a $T$-sequence in $\Z$ whenever $t_\uuu(G)$ contains non-torsion elements} $x$.
Therefore, by Example \ref{Eggl}(d) $t_\uuu(G)$ contains only torsion elements whenever $u_n = P(n)$, for some non-zero polynomial $P(x)\in \mathbb{Z}[x]$
(this was announced without proof in \cite[Example 2.10(a)]{D}).}
\end{rmk}

The possibility to extend the fact that all countable subgroups of $\T$ admit a characterizing sequence to  countable subgroups of arbitrary compact metrizable groups
was very briefly mentioned in \cite{BDS} without saying explicitly what a characterized subgroup and a characterizing sequence must be in the general case.
The necessity to change the pattern $t_\uuu (X)$ used in the case $X = \T$ was pointed out in \cite{DMT} (see also  \cite{D}).
Actually, it was proved in \cite{DDS} that if all cyclic subgroups of a locally compact group  $G$ are intersections of subgroups of the form
(\ref{ee}), then $G\cong \T$, thereby clarifying the fact that the pattern (\ref{ee}) cannot be used for a reasonable definition of characterized subgroup
(for example, if $X$ is the group
of $p$-adic integers for some prime $p$, then $t_\uuu(X)$ coincides with either $X$ or $\{0\}$ \cite[Example 4.11]{D}).

Motivated by the situation described above, the following notion was proposed in \cite{DMT}, making use the subgroups of the form $s_\uuu (X)$ of a topological abelian group $X$:
\begin{definition}{\rm \cite{DMT}} \label{d03}
Let $H$ be a subgroup of a topological abelian group $G$. We say that $H$ is {\em characterized}, if there exists a sequence $\uuu = ( u_n)_{n\in\N}$ in  $G^{\wedge}$ such that  $H= s_\uuu (G)$. In such a case, we say that $\uuu$ {\em characterizes} $H$.
\end{definition}

The importance of the notion of characterized subgroup is obvious  also from the following fact:

\begin{fact} {\rm \cite{DMT}} \label{f1}\label{f13}
Let $\uuu =(u_n)$ be a  $TB$-sequence in an infinite abelian group $G$ and $H=s_\mathbf{u} ((G_d)^\wedge)$.
Then  for every group topology $\tau$ such that $\Tub\leq \tau\leq \Tu$ one has $(G,\tau)^\ast =H$.
In particular, the topology $\tau$ is compatible with $\Tub$ and the sequence $\uuu = (u_n)$ characterizes  $\Tub$.
\end{fact}

Let us point out that in the special case $\T^\wedge = \Z$, one has $s_\uuu (\T) = \{x\in \T: u_n x \to 0\mbox{ in }\T\}=t_\uuu(\T)$, for every sequence $\uuu$ of integers.

The following theorem, which was proved by Kunen and the first named author \cite{DK} and by Beiglb\" ock,  Steineder, and  Winkler \cite{BSW} independently and almost simultaneously, will play a crucial role in our considerations:
\begin{theorem}\label{ThmDK} {\rm \cite{DK}}
Let $H$ be a countable subgroup of a compact metrizable abelian group $G$. Then $H$ is a characterized subgroup.
\end{theorem}

It follows easily from Definition \ref{d03} that every characterized subgroup $H$ is a countable intersection of $F_\sigma$-sets and
hence it is a Borel set. In particular,  characterized subgroups of  compact metrizable groups  can either be countable or has size $\cont$.


\section{Characterizing sequences of the metrizable countable precompact groups}\label{Char2}

We start from the proof of the sufficiency in Theorem A:

 \begin{theorem}\label{MainTheorem}
 Let $G$ be a countable abelian group. Then every metrizable precompact group topology on $G$  has a characterizing sequence.
\end{theorem}

\begin{proof} Let $K=(G_d)^\wedge$ and $\tau$ be a metrizable precompact group topology on $G$. Then $H=(G,\tau)^\ast$ is a dense  countable subgroup of $K$ and $G$ can be identified  with $K^\wedge$ up to isomorphism. By Theorem \ref{ThmDK}, the subgroup $H$ admits a sequence $\uuu =(u_n)$ in $G$ such that
\[
 H =s_\uuu (K).
\]
Since $H$ is dense in $K$, $\uuu$ is a $TB$-sequence by Fact \ref{Crit}.  By Fact \ref{f1}, $(G, \Tub)^\ast = (G,\tau)^\ast$. Since precompact topologies uniquely determined by the set of all continuous characters, we have $\Tub = \tau$. So $\uuu =(u_n)$ characterizes $\tau$.
\end{proof}

The necessity of Theorem A will be proved in \S 3.2.

\begin{rmk}\label{Rem1} {\rm It is relevant to note that, if $G$ is a {\em non-metrizable } countable precompact abelian $ss$-characterized group, then $w(G)= \cont$. Indeed, it is well-known that $w(G)={\rm card}(G^\ast)$. Since $G$  is not metrizable, $w(G)> \aleph_0$. So, by Fact \ref{f1},  $H=s_\mathbf{u} ((G_d)^\wedge)$  is uncountable. As it was noticed after Theorem \ref{ThmDK}, $H$  has size $\cont$. Thus $w(G)=\cont$. }
\end{rmk}

The following fact  and its proof were kindly communicated to us by L. Aussenhofer:

\begin{pro}\label{lo} {\em (L. Aussenhofer)}
Let $G$ be a Hausdorff locally quasi-convex group with discrete Pontryagin dual $G^{\wedge}$.  Then  $G$ is precompact.
\end{pro}

\begin{proof} Since $G$ is a Hausdorff locally quasi-convex group, $\alpha_{G}$ is an injective and open as mapping from $G$ onto $\alpha_G(G)$ by \cite[Proposition 6.10]{Aus}. Since $G^{\wedge}$ is  discrete, its  compact subsets are finite and hence equicontinuous. Thus, $\alpha_{G}$ is continuous  by \cite[Proposition 5.10]{Aus}. Therefore $\alpha_{G}$ is an embedding. So $G$ can be identified with a subgroup of the compact group $G^{\wedge\wedge}$ and hence $G$ is precompact.
\end{proof}

As  an immediate consequence we obtain the following nice and surprising characterization of precompactness for the metrizable abelian groups:

\begin{cor}
Let $G$ be a metrizable abelian group. Then the following are equivalent:
\begin{itemize}
\item[(i)]  $G$ is precompact.
\item[(ii)] $G$ is a locally quasi-convex group with discrete Pontryagin dual $G^{\wedge}$.
\end{itemize}
\end{cor}

\begin{proof}
$(i)\Longrightarrow (ii)$  Let $G$ be precompact. Then its completion $K$ is a metrizable compact group. So the algebraic isomorphism between $G^\wedge$ and $K^\wedge$ is also topological \cite{Aus, Cha}. Hence $G^\wedge$ is discrete. $G$ as a subgroup of the locally quasi-convex group $K$ is locally quasi-convex group itself.

$(ii)\Longrightarrow (i)$ is true by Proposition \ref{lo}.
\end{proof}

The following assertion is of independent interest and prepares the proof of Corollary A.

\begin{pro}\label{p21}
Let $G$ be a countably infinite  abelian group and $\mathbf{u} =(u_n)$ be a $TB$-sequence in $G$ such that $(G, \tau_\uuu)$ is precompact metrizable. Then
\begin{itemize}
\item[{\rm (a)}] $(G, \Tu)^\wedge$ is discrete.
\item[{\rm (b)}] $\tau_\uuu\ne \mathcal T_\uuu$.
\item[{\rm (c)}] $\tau_\uuu =(\Tu)_{\rm{lqc}}$, in particular, $\Tu$ is not locally quasi-convex.
 \end{itemize}
\end{pro}
\begin{proof}
(a) By Fact \ref{f1} we have that $(G, \Tub)^\ast=(G, \Tu)^\ast$. Since $(G, \Tub)$ is precompact
metrizable, $(G, \Tub)^\ast = (G, \Tu)^\ast$ is countable.
By \cite{Ga}, $(G, \Tu)^\wedge$ is a Polish space, so it is discrete as a  countable Polish group.


(b) follows from  Fact \ref{fPre}.

(c) From (a) and the compatibility of $(\Tu)_{\rm{lqc}}$ with $\Tu$ we get that $(G,(\Tu)_{\rm{lqc}})^\wedge$ is discrete. From this by Proposition \ref{lo} we get that $(\Tu)_{\rm{lqc}}$ is precompact.  This implies that $(\Tu)_{\rm{lqc}}= \Tub$, as $\Tub = \Tu^+$.
\end{proof}

\medskip

\noindent {\bf Proof of Corollary A.} By Theorem \ref{MainTheorem}, there exists a $TB$-sequence $\mathbf{u} =(u_n)$ in $G$ such that
$\tau = \tau_\uuu$. As $\mathbf{u} =(u_n)$ is also a $T$-sequence, we can consider the Hausdorff group topology $\eta = \Tu$ on $G$.

 Clearly, $\tau\le \eta$. The topology $\eta $ is compatible with $\tau$ because of $\eta^+ =\tau $ by Fact \ref{f13}. We have $\tau\ne \eta$ by Proposition \ref{p21}(b). By Proposition \ref{p21}(a), $(G,\eta)^\wedge$ is discrete. This implies that $(G,\tau)^\wedge$ is discrete as well, and so $\eta$ is strongly compatible with $\tau$.    $\Box$


\section{Characterizing sequences of arbitrary precompact abelian groups} \label{Char3}

\subsection{The finest precompact extension of a precompact group topology}

For the use in the forthcoming proof of Theorem B, we recall here some well known facts from \cite{CR}. For a discrete group $X$ and a dense subgroup $L$ of the compact dual $(X_d)^\wedge$ we let $T_{L,X}$ (or simply, $T_L$, when no confusion is possible) denote the weak topology $\sigma(X,L)$ of $X$ induced by the subgroup $L$.

\begin{fact} \label{f31}
\begin{itemize}
\item[(a)]  The correspondence $L \mapsto  T_{L,X}$ between dense subgroups $L$ of $(X_d)^\wedge$ and precompact topologies on $X$ is bijective and monotone.
\item[(b)] Let $Y$ be a subgroup of $X$ with inclusion $\iota : Y \hookrightarrow X$, and let $L$ and $M$ be  dense subgroups of $(Y_d)^\wedge$ and $(X_d)^\wedge$
respectively. Then the inclusion $(Y, T_{L,Y})\to (X, T_{M,X})$ is continuous (resp., an embedding) if and only if $\iota^\wedge(M) \subseteq L$ (resp., $\iota^\wedge(M) = L$).
\end{itemize}
\end{fact}

Item (a) comes  from \cite{CR}. To verify (b), it suffices to note, that since both $T_{L,Y}$ and $T_{M,X}$ are weak topologies, the continuity of $(Y, T_{L,Y})\to (X, T_{M,X})$ is equivalent to the fact that $\chi \circ \iota = \chi\restriction_Y$ is continuous whenever $\chi \in  (X, T_{M,X})^\wedge = M$, i.e., $\chi \circ \iota = \iota^\wedge(\chi) \in L= (Y, T_{L,Y})^\wedge$ if $\chi \in M$. This simply means $\iota^\wedge(M) \subseteq L$, as put in (b). The version in brackets is verified similarly (see also the proof of \cite[Theorem 10.1]{DS} for the detailed routine verification of this fact).
\\

\medskip

\noindent{\bf Proof of Theorem B.} We split the verification of (a), namely the equality $(\bar{\zeta})^+=\zeta^*$, in two steps.

 (1){\it The topology $(\bar{\zeta})^+$ is precompact and $(\bar{\zeta})^+|_H={\zeta}$.}\\

 To verify $(\bar{\zeta})^+|_H={\zeta}$, note that $H$ is dually closed and dually embedded in $(G,\bar{\zeta})$, being open in $(G,\bar{\zeta})$.   Hence
$$
(H, \zeta)^\ast=\{\chi|_H: \chi\in  (G,\bar{\zeta})^\ast \}.
$$
From the last equality and Fact \ref{f31}(b), we get that $(\bar{\zeta})^+|_H={\zeta}$. Since the subgroup $H$ is dually closed (i.e., $(\bar{\zeta})^+$-closed), the equality $(\bar{\zeta})^+|_H={\zeta}$ implies that the $(\bar{\zeta})^+$-closure of $\{0\}$ in $G$ coincides with its ${\zeta}$-closure in $H$, namely $\{0\}$. So,  $(G,\bar{\zeta})$ is MAP, hence   $(\bar{\zeta})^+$ is precompact.

(2) {\it Let $\tau'$ be a precompact  topology on $G$ with $\tau'|_H=\zeta$. Then $\tau'\le (\bar{\zeta})^+$.}\\
By Fact \ref{f31}(a), it suffices to show that
\begin{equation}\label{8march}
(G,\tau')^\ast\subseteq (G,\bar{\zeta})^\ast\,.
\end{equation}
So, fix $\chi\in (G,\tau')^\ast$. As $\tau'|_H=\zeta$, we have also that $\chi|_H\in (H,\zeta)^\ast$. Since $H$ is open in $(G,\bar{\zeta})$ from $\chi|_H\in (H,\zeta)^\ast$ and $\bar{\zeta}|_H=\zeta$ we get easily that $\chi\in (G,\bar{\zeta})^\ast$ and (\ref{8march}) is proved.
\par
From (1) and (2) we get that the topology $(\bar{\zeta})^+$ is indeed the finest precompact extension of $\zeta$.
\par

\medskip

(b) The equivalence  (b$_1$) $\Leftrightarrow$ (b$_4$) immediately follows from (a).

For the remaining part of the proof of (b), we need to clarify first the properties of the duals $N=(H,\zeta)^*$ and $Q = (G,\tau)^*$
in view of Fact \ref{f31}. By item (a) of that fact, $N$ and $Q$ are dense subgroups of the compact groups $(H_d)^\wedge$ and $(G_d)^\wedge$, respectively. Let $j: H_d \hookrightarrow G_d$ be the inclusion and $j^\wedge: G_d^\wedge \to H_d^\wedge$ be the corresponding adjoint continuous surjective homomorphism. By Fact \ref{f31}(b), the embedding $(H,\zeta)\to (G,\tau)$ yields
\begin{equation} \label{e*}
j^\wedge(Q) = N.
\end{equation}
Hence, the subgroup $A= (j^\wedge)^{-1} (N)$ of $G_d^\wedge$ contains $Q$ and obvisouly satisfies $j^\wedge(A) = N$. So, by Fact \ref{f31}(b),  the precompact topology $T_A := T_{A,G}$ on $G$ induces on $H$ the original topology $\zeta= T_N$, i.e.,  $T_A$ is an extension of $\zeta$. Moreover, the choice of $A$ implies that $A$ is the largest subgroup of $G_d^\wedge$ with $j^\wedge(A) = N$. Therefore, $T_{A}$ is the finest precompact topology on $G$ inducing on $H$ the original topology $\zeta= T_N$, i.e.,
\begin{equation} \label{e**}
T_{A}=\zeta^\ast .
\end{equation}
At this point we can prove the remaining equivalence of (b).

(b$_1$) $\Leftrightarrow$ (b$_3$). In view of (\ref{e**}), it suffices to note that (b$_1$) means $\tau = \zeta^\ast$, while (b$_3$) means $\tau =T_A$.

To prove {\rm (b$_3$)} $\Leftrightarrow$ {\rm (b$_2$)}, let  $\pi: G \to G/H$ be the canonical map. The continuous characters of the quotient $G/H$, when equipped with the quotient topology of $\tau$, are precisely those  coming from the factorization via $\pi$, of the  $\tau$-continuous characters of $G$ vanishing on $H$, i.e., $\pi^\wedge((G/H)^*) = H^\bot \cap Q.$
Therefore, $H$ is $B$-embedded in $G$ precisely when $H^\bot \subseteq Q$. By (\ref{e*}), $H^\bot \subseteq Q$ is equivalent to
$Q=A$, i.e., $\tau = T_A$. As mentioned above, this is precisely (b$_3$).
$\Box$

The {\it quasi-component} $Q(G)$ of a topological group $G$ is the intersection of all clopen sets of $G$ containing the neutral element of $G$. One can prove that $Q(G)$ is always a normal subgroup of $G$ in the general case when $G$ need not be abelian (\cite{Dz}). A topological group $G$ is

\begin{itemize}
\item {\em totally disconnected} when $Q(G)$ is trivial;
\item {\em hereditarily disconnected} when the connected component $c(G)$ of $G$ is trivial.
\end{itemize}

Since the quasi-component obviously contains the connected component,  totally disconnected groups are hereditarily disconnected.

\begin{pro}\label{c41}
Let  $H$  be a $B$-embedded subgroup  of a precompact abelian group $G$. Then
\begin{itemize}
\item[{\rm (a)}] if $H$ is of infinite index, then $w(G/H) = 2^{|G/H|}$;
\item[{\rm (b)}]  if $H$ is  of infinite index, then $w(G) \geq w(G/H) \geq \cont$;
\item[{\rm (c)}] if $G$ is metrizable, then $G/H$ is finite;
\item[{\rm (d)}]   all subgroups containing $H$ are $B$-embedded as well.
\item[{\rm (e)}]  the subgroup $H$ contains the quasi-connected component $Q(G)$ of $G$.
\end{itemize}
\end{pro}
\begin{proof} (a) is a well-known property of the Bohr topology and follows from Kakutani's theorem $|{\rm{Hom}}(X,\T) |= 2^{|X|}$ for an infinite discrete abelian group $X$.

(b) and (c) follow immediately from (a).

(d) follows from the well-known property of the Bohr topology (every subgroup of $(G_d)^+$ is closed).

(e) It is well-known that abelian groups equipped with their Bohr topology are zero-dimensional (\cite{S}). Since the  quasi-connected component of a zero-dimensional group is trivial, we conclude that $Q(G/H)$ is trivial (here we used the fact that $H$ being $B$-embedded is closed in $G$). Since inverse images of clopen sets along the canonical map $h: G \to G/H$ are still clopen,
we deduce that $ H = \ker h$ is an intersection of clopen sets in $G$, hence $Q(G) \subseteq H$.
\end{proof}

Since countable groups are hereditarily disconnected (actually, zero-dimensional), from Proposition \ref{c41}(e) and Theorem B we deduce

\begin{cor} \label{NEEW}  Every characterized precompact group is hereditarily disconnected.
\end{cor}



\subsection{Proofs of Theorems  A, C  and Corollaries C1, C2, C3}

We start this section from the proof of Proposition \ref{p44}.

\medskip

\noindent{\bf Proof of Proposition \ref{p44}.}
By the definition of $\Tub (H)$, we have $\Tub|_{H} \leq \Tub (H)$. Let $\tau^\ast$ be the finest precompact extension of $\Tub (H)$ that exists by Theorem  B. Clearly, $\uuu\to 0$ in $\tau^\ast$. Thus, by definition, $\tau^\ast \leq\Tub$ and hence $ \tau^\ast|_{H} =\Tub (H) \leq \Tub|_{H}$. So $\Tub|_{H} =\Tub (H)$. This also means that $\Tub$ is an extension of $\Tub (H)$ and hence $\Tub \leq \tau^\ast $. So $\tau^\ast =\Tub$.
$\Box$

\medskip

The following lemma is a folklore fact.

\begin{lemma} \label{l41}
Let $(G,\tau)$ be a precompact abelian group and $H$ be an arbitrary closed subgroup of $G$. Then $H$ is dually closed and dually embedded in $(G,\tau)$.
\end{lemma}

\begin{proof}
Let $\bar{G}$ be the compact completion of $G$. It is well-known that the closure $\bar{H}$ of $H$ is dually closed and dually embedded in $\bar{G}$. Since $\bar{H}$ and $H$ as well as $\bar{G}$ and $G$ have the same set of continuous characters, $H$ is dually closed and dually embedded in $(G,\tau)$.
\end{proof}

As an immediate corollary of Theorem B, Proposition \ref{p44}, Lemma \ref{l41} and Proposition \ref{c41}(d) we obtain:
\begin{pro} \label{p45}
Let $\uuu$ be a $TB$-sequence in an abelian group $G$ and $H$ be an arbitrary subgroup of $G$ containing $\langle\uuu\rangle$. Then
\begin{enumerate}
\item $H$ is dually closed and dually embedded in $(G,\Tub)$;
\item the quotient topology on $G/H$ is the Bohr topology, i.e., $G/H = \left( (G/H)_d \right)^+$.
\end{enumerate}
\end{pro}

\begin{cor} Let $\uuu$ be a $TB$-sequence in an abelian group $G$ and $\tau$ be a group topology on $G$ such that $\Tub \leq \tau$. Then every
subgroup $H$ containing $\langle\uuu\rangle$ is dually closed in $\tau$.
\end{cor}

Taking the trivial sequence $\uuu =\{ 0 \}$ and $H=\{ 0\}$ in Proposition \ref{p45} we obtain a proof of the following well-known fact (cf. \cite[2.1]{CS}):

\begin{cor} \label{c42}
Let $G$ be an abelian group. Then every subgroup of $G^+$ is dually closed and dually embedded.
\end{cor}

\medskip

\noindent {\bf Proof of Theorem C.} ${\rm (i)} \Rightarrow {\rm (ii)}$ Let $\uuu$ be a $TB$-sequence in $G$ which characterizes $\tau$.
Putting $H=\langle\uuu\rangle$ assertion (ii) follows from (i) by Theorem B and Proposition \ref{p44}.

${\rm (ii)} \Rightarrow {\rm (i)}$ Let  $H$ be a countable $B$-embedded subgroup of $G$ and let $\uuu$ be a $TB$-sequence such that $(H,\tau|_H)=(H, \Tub(H))$. Let us show that $\tau =\Tub$.

The hypothesis (ii) means that  $\tau = (\tau_\mathbf{u}(H))^\ast$ is the finest  precompact extension of $\tau_\mathbf{u}(H)$. Then Proposition \ref{p44} implies that $\tau =\Tub$. $\Box$

\bigskip

\noindent {\bf Proof of Theorem A.} Assume that $G$ is $ss$-characterized. According to Theorem B, $G$ has a countable
$B$-embedded subgroup $H$. By Proposition \ref{c41}(c), $H$ must have finite index in $G$.  Thus $G$ is countable as well.

Conversely, if $G$ is countable, then $G$ is $ss$-characterized by Theorem \ref{MainTheorem}.

Let  $\uuu$ be a sequence which  characterizes $G$. By Proposition \ref{p45} $G/\langle\uuu\rangle$ carries the Bohr topology and it is
metrizable. By Proposition \ref{c41}(c), $\langle\uuu\rangle$ must have finite index. $\Box$

\bigskip

\noindent {\bf Proof of Corollary C1.} Assume that $G$ is $ss$-characterized. According to Theorem C, $G$ has a countable $B$-embedded subgroup $H$. So the group $G/H$ carries the Bohr topology. Since $G/H$ is also pseudocompact (as a quotient of $G$), it follows that $G/H$ is finite \cite{CS}. Hence $G$ is countable. Therefore $G$ is a countable pseudocompact group and hence it must be finite (as infinite pseudocompact groups are uncountable \cite{vanD}).

If $G$ is finite it is characterized by the trivial sequence $\uuu =(0)$. $\Box$

\begin{lemma} \label{l*}
Let $X$ be a Hausdorff countable space. Then $X$ is a $k$-space if and only if it is sequential.
\end{lemma}

\begin{proof} It is well known that every sequential space is a $k$-space \cite[3.3.20]{Eng}.

Let $X$ be a $k$-space and $K$ a compact subset of $X$. Being countable, $K$ is metrizable \cite[3.1.21]{Eng}. Thus $X$ is sequential by Lemma 1.5 of \cite{CMT2}.
\end{proof}

\medskip

\noindent {\bf Proof of Corollary C2.} Let $(G,\tau)$ be $ss$-characterized precompact  group, for which $(G,\tau)$ is a $k$-space. Take a sequence $\uuu$ in $G$  which  characterizes $\tau$ and set $H=\langle \uuu \rangle$. Then the quotient  group $(G/H, \tau/H)$ carries the Bohr topology by Theorem C. Thus $(G/H, \tau/H)$ has no infinite compact subsets. From this, since $(G/H, \tau/H)$ is a $k$-space as well, we get  that $(G/H, \tau/H)$ is discrete. Therefore, $(G/H, \tau/H)$ is a discrete precompact group. Hence, $G/H$ is finite. Since $H$ is countable, we get that $G$ is countable as well. By Lemma \ref{l*}, $G$ is sequential.
$\Box$

\bigskip

\noindent {\bf Proof of Corollary C3.} Metrizable precompact groups have countable weight, so they are Arhangel$'$ski\u \i \ groups.

Now assume that $G$ is an uncountable $ss$-characterized precompact group. By Theorem C, there exists a $B$-embedded countable subgroup
$H$ of $G$. Then $|G/H| = |G|$, so by Proposition  \ref{c41} $w(G/H) = 2^{|G/H|}=2^{|G|} > |G|$. Since $w(G) \geq w(G/H)$, this proves that $w(G) > |G|$, i.e., $G $ is not   an Arhangel$'$ski\u \i \ group.
$\Box$

\subsection{Sequential completeness} According to \cite{DT1}, a topological group $G$ is said to be {\em
sequentially complete}, if every Cauchy sequence in $G$ is convergent (or, equivalently, when $G$ is sequentially closed in its two-sided completion).  For basic properties of sequentially complete groups see \cite{DT1,DT2}.

By Remark \ref{zprot} for every $T$-sequence $\uuu$ in an infinite abelian group $G$ the group $(G,\Tu)$ is sequential and complete. By Corollary C1,  if $G$ is an infinite abelian group  and $\uuu$ is a $TB$-sequence in $G$, then the group $(G,\Tub)$ cannot be complete. So it is natural to ask whether the group $(G,\Tub)$ is sequentially complete.  In spite of the property to be sequential complete is not a three space property in general \cite{BT}, we can prove the following.

\begin{pro} \label{sc}
Let $H$  be a closed subgroup of a topological abelian group $G$. If $G/H$ has no non-trivial convergent sequences, then $G$ is sequentially complete if and only if $H$ is sequentially complete.
\end{pro}

\begin{proof} Assume that $H$ is sequentially complete. Let $\bar{G}$ be the completion of $G$, $\bar{H} =\mathrm{cl}_{\bar{G}} (H)$ and let $q:\bar{G} \to \bar{G}/\bar{H}$ be the quotient map. To check that $G$ is sequentially closed in $\bar{G}$ pick a sequence $\vvv =(v_n)$ in $G$ converging to an element $\bar{g}\in \bar{G}$. Then $q(\vvv)$ converges to $q(\bar{g})$. By hypothesis, $q(\vvv)$ is trivial. Hence there exists  $k \in \N$ such that $q(v_n) = q(v_k)$  for all $n \geq k$. Hence $h_n:= v_n - v_k\in
\bar{H}$ for all $n \geq k$. As $h_n \in G$ as well, we deduce that $ h_n \in G \cap \bar{H} =H$. As $h_n $ converges to $\bar{g} - v_k$
in $\bar{H} $ and $H$ is sequentially complete, we have $\bar{g} - v_k\in H$. Thus $\bar{g}\in G$. Therefore $G$ is sequentially closed in $\bar{G}$.

If $G$ is sequentially complete, then $H$, as a closed subgroup of $G$, is sequentially complete as well.
\end{proof}

\begin{cor} \label{Feb12} Let $\uuu$ be a $TB$-sequences in an abelian group $G$. Then $(G,\Tub)$ is sequentially complete if and only if the countable
subgroup $(\langle\uuu\rangle,\Tub|_{\langle\uuu\rangle})$ is sequentially complete.
\end{cor}

We are not aware whether such groups can be sequentially complete (see Problem \ref{seq_complete}).

\subsection{$T$-sequences and local quasi-convexity}

Let $\uuu$ be a $TB$-sequence in a group $G$.
 Since $\Tub$ is locally quasi-convex, we have $\Tub\leq (\Tu)_{\mathrm{lqc}}$.

Now let $\mathbf{u}$ and $\mathbf{v}$ be  $TB$-sequences in an infinite
abelian group $G$. By  Fact \ref{f13}, if $\mathcal T_{\bf u}
=\mathcal T_{\bf v}$, then $\tau_{\bf u} =\tau_{\bf v}$. The next
theorem gives a particular answer to the question whether the
converse assertion is true.
\begin{theorem} \label{t42}
Let $\mathbf{u}$ and $\mathbf{v}$ be  $TB$-sequences in an infinite
abelian group $G$. Assume that
\begin{enumerate}
\item[{\rm 1.}] $\Tu$ and $\Tv$ are locally quasi-convex;
\item[{\rm 2.}]$\Tub =\Tvb$ .
\end{enumerate}
Then $\Tu=\Tv$.
\end{theorem}

\begin{proof}
Set $H=\langle \uuu\rangle+\langle {\bf v}\rangle$. If $H$ is finite, the theorem is trivial. Let $H$ be countably infinite. Since
$H$ is open in both the topologies $\Tu$ and $\Tv$ and since $\Tub(H) =\Tvb(H)$ by Corollary \ref{c45}, we can assume that $G$ is countably infinite.

Let $i_\mathbf{u} : G_d \to (G,\Tu)$ and $i_\mathbf{v} : G_d \to (G,\Tv)$ be the identity maps. Then, by item 2 and Fact \ref{f13}, $i_\mathbf{u}^\wedge (G)= i_\mathbf{v}^\wedge (G):=H$ as a subgroup of the compact metrizable group $(G_d)^\wedge$. By \cite{Ga}, $(G,\Tu)^\wedge$ and $(G,\Tv)^\wedge$ are Polish groups. So, the Borel subgroup $H$ of
$(G_d)^\wedge$ admits two finer Polish group topologies. By the uniqueness of Polish group topology, we obtain that $(G,\Tu)^\wedge =(G,\Tv)^\wedge$ topologically. Hence
\[
(G,\Tu)^{\wedge\wedge} =(G,\Tv)^{\wedge\wedge} \mbox{
topologically and } \alpha_{(G,\Tu)} =\alpha_{(G,\Tv)} \mbox{ as algebraic homomorphisms}.
\]
Since $\Tu$ and $\Tv$ are locally quasi-convex and  $k$-spaces by Remark \ref{zprot},
$\alpha_{(G,\Tu)}$ and $\alpha_{(G,\Tv)}$ are embedding by \cite[5.12 and 6.10]{Aus}. Thus $\Tu =\Tv$.
\end{proof}

\section{Final remarks and problems}

Clearly, if a non-trivial $TB$-sequence $\uuu$ in an infinite abelian group $G$ is such that $(\langle\uuu\rangle, \Tub)$ is
metrizable, the countably infinite group $(\langle\uuu\rangle, \Tub)$ is not sequentially complete. By Proposition
\ref{p45} $(\langle\uuu\rangle, \Tub)$ is closed in $(G,\Tub)$. So $(G,\Tub)$ is not sequentially complete by Proposition \ref{sc}.
Therefore, the following question is of independent interest.

\begin{question}\label{seq_complete}
\begin{itemize}
\item[(a)] Is there countably infinite sequentially complete  precompact $ss$-characterized group{\rm ?}
 \item[(b)] In particular, if $\uuu$ is the sequence $(n!)$ in $\Z$, is $(\Z, \tau_\uuu)$ sequentially complete{\rm ?}
\end{itemize}
\end{question}

The assumption on $\Tu$ and $\Tv$ to be locally quasi-convex in Theorem \ref{t42} is essential. As it was noticed in
\cite{Ga3}, in general, the converse assertion is not true, i.e., from the equality $\Tub =\Tvb$ it does not follow that $\Tu
=\Tv$. More precisely, in that example $\bf v$ converges to zero in $\Tub$, but $\bf v$ is not a null-sequence in
$\Tu$. So $\Tub$ may have essentially more converging sequences than $\Tu$. We end this discussion by the following questions.

\begin{question}
Let $\uuu$  be a $TB$-sequence in an infinite abelian group $G$. Do the groups $(G,\Tub)$ and $\left(G, (\Tu)_{\mathrm{lqc}}\right)$ have the same null sequences (or compact subsets){\rm ?}
\end{question}

\begin{question}
Let $\uuu$  be a $TB$-sequence in an infinite abelian group $G$ such that $(G,\Tub)$ is not metrizable. Does there exist a LCQ-group topology $\tau$ on $G$ such that   $\Tub < \tau <(\Tu)_{\mathrm{lqc}}$ (cf. Proposition \ref{p21}){\rm ?}
\end{question}

\begin{question}
Let $\uuu$ and $\bf v$ be $TB$-sequences in an infinite abelian group $G$ such that $\Tub = \Tvb$. Is $(\Tu)_{\mathrm{lqc}} =(\mathcal
T_{\bf v})_{\mathrm{lqc}}${\rm ?}
\end{question}

The next problem should be compared with Proposition \ref{c41} and Corollary \ref{NEEW}:

\begin{question} \label{p_today}
Prove or disprove that every $ss-$characterized precompact group is totally disconnected (zero-dimenional).
\end{question}

 Let us note that the difference between totally disconnected and hereditarily disconnected groups  is very subtle. While always the implications\\

\centerline{zero-dimenional \ $\Longrightarrow $	 \ totally disconnected \ $\Longrightarrow$ \ hereditarily disconnected}

\medskip

\noindent hold true, the inverse implications may fail in general (but hold true in locally compact or in countably compact groups
\cite{D0}). A careful analysis of the short argument outlined in front of Corollary \ref{NEEW} shows that
$Q(Q(G)) = 0$ for every $ss$-characterized precompact group. This is stronger than the current statement of that corollary, since the ordinal chain
\begin{equation} \label{d}
Q(G) \supseteq Q(Q(G)) \supseteq Q(Q(Q(G)))  \supseteq \ldots
\end{equation}
has as intersection the connected component $c(G)$ of the group $G$ (note that for every ordinal $\alpha$ one may find even a pseudocompact abelian group $G_\alpha$ for which the ordinal chain (\ref{d}) has length exactly $ \alpha$ \cite{D-1}). Hence, in some sense the actual result $Q(Q(G)) = 0$ gives a good evidence that $ss$-characterized precompact group may be totally disconnected.

The next questions are motivated by Corollary C2 and by the fact
that the topologies determined by a non-trivial $T$-sequence are
always sequential, but never Fr\'echet-Urysohn (see Remark
\ref{zprot})

\begin{question} \label{p2}
Let $(G,\tau)$ be a countable precompact ss-characterized group. Can $(G,\tau)$
be sequential, but not Fr\'echet-Urysohn{\rm ?} Can $(G,\tau)$ be
Fr\'echet-Urysohn, but not metrizable{\rm ?}
\end{question}

\end{document}